\documentclass{amsart}
\usepackage{color}
\usepackage{url}
\usepackage{amssymb,url}
\usepackage[foot]{amsaddr}
\usepackage[norelsize,lined,boxed,commentsnumbered]{algorithm2e}

\title
[Inverse moment problem for convex polytopes]
{The inverse moment problem for convex polytopes}

\author[Nick Gravin, Jean Lasserre, Dmitrii Pasechnik, Sinai Robins]
{Nick Gravin$^{1,3}$,  Jean Lasserre$^2$,
Dmitrii V. Pasechnik$^1$, Sinai Robins$^1$}
\address{$^1$ School of Physical and Mathematical Sciences, Nanyang
Technological University, 21 Nanyang Link, 637371 Singapore. {\rm Supported by Singapore Ministry of Education ARF Tier 2 Grant MOE2011-T2-1-090.}}
\address{$^2$  LAAS-CNRS, 7 Avenue du Colonel Roche,
31 077 Toulouse Cedex 4, France}
\address{$^3$  Steklov Institute of Mathematics at St.Petersburg,
27 Fontanka, St.Petersburg 191023, Russia}

\date\today

\theoremstyle{definition}
\newtheorem{prop}{Proposition}[section]
\newtheorem{theorem}{Theorem}

\newtheorem{lemma}{Lemma}

\newtheorem*{mthm}{Main Theorem}

\newtheorem{rem}[prop]{Remark}

\theoremstyle{remark}
\newcounter{reml}

\newcommand{\C}{{\mathbb{C}}}
\newcommand{\F}{{\mathbb{F}}}

\newcommand{\K}{{\text{Ker}(\mathbf{H})}}

\newcommand{\Q}{{\mathbb{Q}}}
\newcommand{\R}{{\mathbb{R}}}
\newcommand{\V}{{\text{Vert}}}
\newcommand{\Z}{{\mathbb{Z}}}

\DeclareMathOperator{\vol}{Vol}

\renewcommand{\epsilon}{\varepsilon}
\renewcommand{\theta}{\vartheta}

\newcommand{\w}{\mathbf{w}}
\newcommand{\ve}{\mathbf{v}}
\newcommand{\x}{\mathbf{x}}
\newcommand{\y}{\mathbf{y}}
\newcommand{\z}{\mathbf{z}}
\newcommand{\m}{\mathbf{m}}

\newcommand{\de}{{d^o}}
\newcommand{\pr}{\mathbf{Prob}}

\newcommand{\ba}{\mathbf{a}}
\newcommand{\bb}{\mathbf{b}}
\newcommand{\bc}{\mathbf{c}}
\newcommand{\be}{\mathbf{e}}
\newcommand{\bu}{\mathbf{u}}
\newcommand{\bv}{\mathbf{v}}

\newcommand{\bD}{\mathbf{D}}
\newcommand{\bH}{\mathbf{H}}
\newcommand{\bT}{\mathbf{T}}
\newcommand{\bV}{\mathbf{V}}

\newcommand{\bil}[2]{\langle{#1},{#2}\rangle}

\newcommand{\marginnote}[1]{\vrule width0pt height0pt depth0pt
 \vadjust{\vbox to0pt{\vss\hbox to\hsize{\hskip\hsize\quad
 #1\hss}\vskip1.5pt}}}

\newcommand{\beq}{\begin{equation}}
\newcommand{\eeq}{\end{equation}}

\newcommand{\BrLafull}{Brion-Barvinok-Khovanskii-Lawrence-Pukhlikov}
\newcommand{\BrLa}{BBaKLP}

\begin{document}
\begin{abstract}
We present a general and novel approach for
the reconstruction of any convex $d$-dimensional polytope
$P$, assuming knowledge of finitely many of its integral
moments. In particular, we show that the vertices of an
$N$-vertex convex polytope in $\R^d$ can be reconstructed
from the knowledge of $O(DN)$ axial moments (w.r.t. to an
unknown polynomial measure of degree $D$), in $d+1$ distinct
directions in general position. Our approach is based on the collection
of moment formulas due to Brion, Lawrence,
Khovanskii-Pukhikov, and Barvinok that arise in the discrete
geometry of polytopes, combined with what is variously known
as Prony's method, or the Vandermonde factorization of finite
rank Hankel matrices.
\end{abstract}

\maketitle

\section{Introduction}
\label{sec:intro}

The {\em inverse} problem of recognizing an
object from its given moments is a fundamental
and important problem in both applied and pure mathematics. For example, this
problem arises quite often in computer tomography, inverse potentials, signal
processing, and statistics and probability.  In computer tomography, for
instance, the X-ray images of an object can be used to estimate the moments of
the underlying mass distribution, from which one seeks to recover the shape of
the object that appear on some given images. In gravimetry applications, the
measurements of the gravitational field can be converted into information
concerning the moments, from which one seeks to recover the shape of the source
of the anomaly.

The goal of this paper is to present a general and novel approach for the
reconstruction of any convex $d$-dimensional polytope $P$, from knowledge of its moments.  Our approach is quite different
from the quadrature-based approach that is currently used in the  literature.    Our starting point
is the collection of moment formulas, due to
\BrLafull\  (in what follows referred for brevity as \BrLa )
that arises in the discrete
geometry of polytopes, and is valid for all dimensions \cite{MR982338,MR1079024,MR1173086,MR1118837}, and \cite[Chapter~10]{BR07}.   
We then set up a matrix equation involving a variable Vandermonde matrix,
with an associated Hankel matrix whose kernel helps us reconstruct the vertices of $P$.  We are also able to reconstruct the vertices of a convex polytope
with variable density,  using similar methods.
\nocite{MR1190788}

This new approach permits us to reconstruct {\em exactly} the vertices of $P$, using very few moments, relative to the vertex description of $P$.   While the  computation of integrals over polytopes has received attention recently (see e.g.
\cite{MR2728981}), our work appears to be the first to provide tools to treat the inverse moment problem in general.

A nice feature of our algorithm is that we do not need to know a priori the
number of vertices of $P$, only a rough upper bound for their number.  The
algorithm automatically retrieves the number of vertices of $P$ as the rank of
a certain explicit Hankel matrix.

In fact, a surprising corollary is that we only require $O(Nd)$ moments in order to reconstruct all of the $N$ vertices of $P \subset \R^d$.
Suppose we are solving the inverse moment problem in the context of an unknown density function $\rho$.
An interesting consequence of our algorithm is that even though $\rho$ is unknown, we may easily adapt our algorithm to recover the vertices of $P$ in $O(N \de d)$ steps, where $\de$ is an upper bound on the degree of $\rho$.

Now suppose we simply solve the direct problem of writing down the moments of a given vertex set of a known polytope, with a known polynomial density function $\rho$.  In this direct problem, it would take $\binom{\de + d}{d}$ data to describe the polynomial  $\rho$ function, because the space of possible polynomials $\rho$ has this dimension.  However, for the inverse problem, where $\rho$ is unknown, perhaps a counter-intuitive consequence of our algorithm is that
we only require $O(N \de d)$ data to recover the vertex set of $P$, which might be smaller than $\binom{\de + d}{d}$.

In the existing literature on inverse problems from moments, one immediately
encounters a sharp distinction between the $2$-dimensional case and the general
$d$-dimensional case, with $d>2$. While in the former case
a well-known quadrature formula allows us to solve the problem exactly for
so-called quadrature domains, where for the latter case one has to ``slice up''
the domain of interest into thin 2-dimensional pieces, solve the resulting
2-dimensional problems, and patch up an approximate solution from these
2-dimensional solutions. On the other hand, in the recent  work of Cuyt et al.
\cite{MR2199920} the authors can approximately recover a general
$n$-dimensional shape by using an interesting property of multi-dimensional
Pad\'e approximants.


For $\z\in\R^d$ and each nonnegative integer $j$, we define the $j$-th {\em moment} of $P$ with respect to the density $\rho$ by:
\[
\mu_j(\z):= \mu_{j,\rho} (\z) := \int_P  \langle \x, \z \rangle  ^j \rho(\x) d\x.
 \]
In this text we restrict ourselves to any density function $\rho$ which is given by a polynomial measure, and which does not vanish
on the vertices $\V(P)$ of $P$.

 We note that only in the Appendix, when we give proofs
of the known \BrLa\  moment formulas  below, we will need to replace the real
vector $\z$ by a complex vector, in order to allow convergence of some
Fourier-Laplace transforms of cones, but otherwise $\z$ will always be a real
vector.     We say that $\z$ is in {\it general position} if it is chosen at random from
the continuous Gaussian distribution on $\R^d$. 

\medskip
\noindent
Our main result may be formulated as
follows.
\newpage
\begin{mthm}\label{thm:mthm}
Let $P\subset \R^d$ be a $d$-dimensional polytope with $N$ vertices, and suppose we are only given the data in the form of
$O(d_\rho N)$ moments $\mu_{j,\rho}(\z)$,
for an unknown density $\rho\in\R[\x]$ of degree $d_\rho$, and for
each of $d+1$ vectors $\z\in\R^d$ in general position.
Then the data determines $P$ uniquely, using the following algorithm:

\begin{enumerate}

\item  Given $2m-1 \ge 2N+1$ moments $c_1,\dots,c_{2m-1}$ for $\z$, construct a square Hankel matrix $\bH(c_1,\dots,c_{2m-1}).$
\medskip
\item Find the vector $v=\left(a_0, \ldots, a_{M-1}, 1, 0, \ldots, 0 \right)$ in $\K$ with the minimal possible $M.$ It turns out that the number of 
    vertices $N$ is in fact equal to $M$.
\medskip
\item The set of roots $\{x_i(\z) = \bil{\bv_i}{\z} |\bv_i\in\V(P)\}$ of the polynomial
      $p_\z(t) =t^N+\sum_{i=0}^{N-1}a_i t^i$
       then equals the set of  projections of $\V(P)$ onto $\z$.
\end{enumerate}

By contrast with a choice of a general position vector $\z$, we also define, for a simple polytope $P$,
a {\it generic vector} $\z \in \Q^d$ to be a vector that lies in the complement of the finite union of
hyperplanes which are orthogonal to all of the edges of $P$.  For non-simple polytopes $P$, we
will later extend this definition of a generic vector $\z \in \Q^d$, in Section \ref{sect:rational.implementation}.

We furthermore prove in Section \ref{sec:univar} that the vertex set
$\V(P)\subset\Q^d$ of any rational convex polytope $P$ can be found in polynomial time with a probability arbitrary close to $1$, 
from the exact measurements of $O(d_\rho N)$ moments in carefully chosen $2d-1$ random generic directions $\z\in\Q^d$.

In Section \ref{sec:algorithm}, we indicate how $\V(P)\subset\R^d$  can also be efficiently approximated even when the data is noisy.
\end{mthm}

One punchline of the proof is that an appropriate scaling of the
sequence of  the moments $\mu_{j,\rho}(\z)$ ($j=0,1,\dots$) for a fixed $\z$
is a finite sum of exponential functions, and thus
satisfies a linear recurrence relation
(cf. e.g. \cite[Theorem 4.1.1(iii)]{MR1442260}). Then an application
of what is variously known as Prony's method, or Vandermonde factorization
of a finite rank Hankel matrix (cf. e.g. \cite{MR1729647}),
allows one to find $\bil{\z}{\ve}$ for $\ve\in\V(P)$.
As these methods are scattered along quite a number of sources,
we have chosen to present a self-contained exposition for clarity and for ease of efficient implementation.

Reconstructing $\V(P)$ from the  $\bil{\z}{\ve}$ is then relatively
straightforward, provided that we know these projections for sufficiently many
$\z$ in general position.  For the latter, we present an exact procedure as well as a
parametric one---the latter with the focus being less noise-sensitive.

The remainder of the paper is organized as follows.   In Section
\ref{sec:prelim}, we define the objects we are dealing with, as well as
the appropriate background for ease of reading.  We also give
the known formulas for the moments of simple polytopes.  In Section
\ref{sec:project} we construct a polynomial whose roots correspond to the
projections of the vertices onto directions $\z$ in general positions, by using the moment
formulas and an associated Hankel matrix.
In Section~\ref{sec:polymeasure} we extend the latter to the case of
unknown polynomial measures.  In Sections
\ref{sec:non_simple} and
\ref{sec:algorithm}, we
extend the algorithm from Sections \ref{sec:project} and \ref{sec:polymeasure},
which deals with simple
polytopes, to all convex polytopes. This completes the proof of the first claim
of Main Theorem. In Section \ref{sec:polymeasure} we also discuss the question of
reconstructing $\rho$ after $\V(P)$ is found.

In Section \ref{sec:univar} we use
univariate polynomials to paste together the projections retrieved from Section
\ref{sec:project} to build up all of the coordinates of each vertex, not just
their projections, completing the proof of the second claim of
Main Theorem.  In Appendix \ref{LawrenceProof}
we outline
some proofs of the known \BrLa\ moment formulas from Section \ref{sec:prelim},
using Fourier techniques.

\section{Definitions and moment formulas for convex, rational polytopes}
\label{sec:prelim}

Here we describe an explicit set of formulas for the moments of any  convex
polytope $P \subset \R^d$.  We begin with some combinatorial-geometric
definitions of the objects involved.  To fix notation, our convex polytope $P$
will always have $\mathit{N}$ vertices.  We say that $P$ is {\em simple} if each vertex
$\bv$ of $P$ is incident with exactly $d$ edges of $P$.   We first treat the
case of a  simple convex polytope
and then later, in Section \ref{sec:non_simple} we provide an extension to non-simple convex polytopes.

There is an elegant and useful formulation, originally due to \BrLa\
\cite{MR1079024}, for the moments of any simple polytope in $\R^d$, in terms of
its vertex and edge data.  Specifically, let the set of all vertices of $P$ be
given by $\V(P)$.   For each $\bv \in \V(P)$, we consider a fixed set of
vectors, parallel to the edges of $P$ that are incident with $\bv$, and call
these edge vectors $w_1(\bv)$,\dots $w_d(\bv)$.  Geometrically, the polyhedral
cone generated by the non-negative real span of these edges at $\bv$ is called
the tangent cone at $\bv$, and is written as $K_\bv$.  For each simple tangent
cone $K_\bv$, we let $|\det K_\bv|$ be the volume of the parallelepiped formed
by the $d$ edge vectors $w_1(\bv), \dots, w_d(\bv)$.  Thus, $|\det K_\bv| = |
\det( w_1(\bv),  \dots, w_d(\bv)) |$, the determinant of this parallelepiped.

The following results of \BrLa\ \cite{MR1079024} give  the moments
of a simple polytope $P$ in terms of the local vertex and tangent
cone data that we described above. For each integer $j \geq 0$, we
have
\beq\label{brion} \mu_j (\z)= \frac{j! (-1)^d}{ (j+d)!}
\sum_{\bv\in \V(P)} \bil{\bv}{\z}^{j+d} D_\bv(\z),
\eeq
where
\beq\label{Dvz} D_\bv(\z):=\frac{|\det
K_\bv|}{\prod_{k=1}^d\bil{w_k(\bv)}{\z}},
\eeq
for each $\z \in
\R^d$ such that the denominators in $D_\bv(\z)$ do not vanish.
Moreover, we also have the following companion identities:
\beq\label{brionzero} 0=\sum_{v\in \V(P)} \bil{\bv}{\z}^{j}
D_\bv(\z), \eeq for each $ 0 \leq j \leq d-1$.    Thus, for
example, if $\z=(1, 0, \dots , 0)$, the equations \eqref{brion}
and \eqref{brionzero} deal with the first coordinate of each of
the vertices $\bv \in \V(P)$.   We also note that all of these
formulas involve only homogeneous, rational functions of $\z=(z_1,
\dots, z_d)$.

In the more general case of non-simple polytopes, we may triangulate each tangent cone into simple cones, thereby
getting a slightly more general form of \eqref{brion} above, namely:
\beq \mu_j (\z)= \frac{j! (-1)^d}{ (j+d)!}
\sum_{\bv\in \V(P)} \bil{\bv}{\z}^{j+d} \tilde{D}_\bv(\z), \eeq where
each $\tilde{D}_\bv(\z)$ is a rational function that now comes from the non-simple tangent cone at $\bv$.
We note that $\tilde{D}_\bv(\z)$ is in fact a sum of the relevant  rational functions $D_\bv(\z)$
that are associated
to each simple cone in the triangulation of the non-simple tangent cone $K_{\bv}$.

Throughout the paper, we will mainly work in the context of a continuous domain for our choices of
admissible $\z$ vectors.  To be precise, we say that a vector $\z$ is in {\it general position}
 if none of the following conditions is true:
\begin{enumerate}
\item[(a)]  $\z$ is a zero or a pole of $\tilde{D}_\bv(\z)$ , for any $\bv \in \V(P)$.
\item[(b)]  There exist two vertices $\bv_1, \bv_2 \in \V(P)$ such that
 $\bil{\bv_1}{\z}=\bil{\bv_2}{\z}$.
 \end{enumerate}
In the penultimate section, we indicate how to implement our algorithm by transitioning
all of our formulas to a rational context, and picking our $\z$ vectors to lie in a finite,
rational $d$-dimensional cube.

The goal here is to reconstruct  the  polytope $P$ from a  given sequence of  moments $\{ \mu_j (\z) \ |  \ j = 1, 2, 3, \ldots \}$.
That is, we wish to find an explicit algorithm that locates the vertices $v$ of the polytope $P$, in terms of the moments of $P$.
To emphasize the fact that the moment equations above can be put into matrix form, we define our scaled vector of moments by:
\beq\label{c-vector}
\left (c_1, \dots, c_{k+1} \right) = \left(0, \dots, 0, \frac{d!(-1)^d}{0!} \mu_0,  \frac{(1+d)!(-1)^d}{1!} \mu_1, \dots, \frac{k!(-1)^d}{(k-d)!} \mu_{k-d}\right),
\eeq
so that the vector ${\bf{c}} = \left(c_1, \dots, c_{k +1}\right)$ has zeros in the first $d$ coordinates, and scaled moments in the last $k+1-d$ coordinates.
Thus, putting the moment identities \eqref{brion} and  \eqref{brionzero} above into matrix form, we have:

\beq\label{momentmatrix}
\begin{pmatrix}
1&1&\dots&1\\
 \langle \bv_1, \z \rangle &  \langle \bv_2, \z \rangle & \dots &  \langle \bv_N, \z \rangle   \\
{ \langle \bv_1, \z \rangle}^2 & { \langle \bv_2, \z \rangle}^2 & \dots &  {\langle \bv_N, \z \rangle}^2  \\
\vdots&\vdots&\dots&\vdots \\
  { \langle \bv_1, \z \rangle}^k & { \langle \bv_2, \z \rangle}^k & \dots &  {\langle \bv_N, \z \rangle}^k  \\
\end{pmatrix}
\begin{pmatrix}
D_{\bv_1}(\z)\\
\vdots\\
D_{\bv_N}(\z)
\end{pmatrix}=
\begin{pmatrix}
c_1\\
\vdots\\
c_{k+1}
\end{pmatrix}.
\eeq

Recalling that we  seek to find the vertices of the convex polytope $P$ with these given scaled moments $c_i$, we
 answer this question completely, giving an efficient algorithm to recover the vertices of such an object $P$.

We will treat each $D_{\bv_i}(\z)$ as a nonzero constant, which we have not yet
discovered, and each $ \langle \bv_j, \z \rangle$ as a variable, for a fixed
real vector $\z \in \R^d$.  Moreover, we realize below that, given our
algorithm, only finitely many moments $\mu_1(\z), \dots, \mu_M(\z)$ are needed
in order to completely recover the full vertex set $\V(P)$, a rather useful
fact for applications.

We recall that our $j$'th moment of $P$ is defined, for the uniform measure $\rho \equiv 1$, by
\beq\label{eq:defmomz}
\mu_j (\z) = \int_P  \langle \x, \z \rangle  ^j d\x
\eeq
We note that $\mu_0(\z) = \vol(P)$, the volume of $P$ with respect to the usual
Lebesgue measure.

It is very natural to study moments in this form, because they are ``basis-free" and also appear as moments of inertia in physical applications.
It is worth noting that there are other types of moments in the literature, and we mention some connections here.
For each integer vector $\m$, we define
\beq\label{eq:defmomm}
\mu_{\m} = \int_P  \x^{\m} d\x,
\eeq

\noindent
with the usual convention that $ \x^{\m}=\prod_{i=1}^d x_i^{m_i}$ and $|\m|=m_1+\dots+m_d$.
The usual application of the binomial theorem gives us a trivial relation between these moments:
\begin{equation}
 {\langle \z, \x \rangle}^{k}=
\sum_{\substack{m_1,\dots,m_d:\\ m_1+\dots+m_d=k}}
\begin{pmatrix}
k\\
m_1, \dots, m_d
\end{pmatrix}
z_1^{m_1} \cdots z_d^{m_d}
x_1^{m_1} \cdots x_d^{m_d}.
\end{equation}

In fact, given $\V(P)$ and $\mu_{|\m|}(\z)$ as a function of $\z$, we can also
compute $\mu_{\m}$, as following Lemma shows.   Its proof is trivial , but it nevertheless offers an
interesting relation between the moments  \eqref{eq:defmomz}  and
 \eqref{eq:defmomm}.


\begin{lemma}\label{lem:momzm}
Let $\m\in\Z_+^d$, $\V(P)$,  and $\mu_{|\m|}(\z)$ be given.  Then
\[ |m|!\mu_{\m}=\frac{\partial^{|\m|}}{\partial \z^{\m}}\mu_{|\m|}(\z). \qedhere\]
\end{lemma}

\section{The inverse moment problem for polytopes -  computing projections of $\V(P)$}
\label{sec:project}

In this section we show how, for a given general position vector $\z$,
to retrieve the
projections $\bil{\bv}{\z}$ for each vertex $\bv$ of $P$, 
using a certain Hankel matrix that we define below.
For the sake of convenience we
let  $x_i = \bil{\bv_i}{\z}$, for each
$1\leq i\leq N$.  Thus, our goal for this section is to find all $x_i$, given
a number of moments. Due to our choice of $\z$, we may assume that
 $x_i\neq x_j$ for $i\neq j$.  From \eqref{momentmatrix} we have
\beq\label{vanderm}
\begin{pmatrix}
1&1&\dots&1\\
x_1&x_2&\dots&x_N\\
x^{2}_{1}&x^{2}_{2}&\dots&x^{2}_{N}\\
\vdots&\vdots&\dots&\vdots\\
x^{k}_{1}&x^{k}_{2}&\dots&x^{k}_{N}\\
\end{pmatrix}
\begin{pmatrix}
D_{\bv_1}(\z)\\
\vdots\\
D_{\bv_N}(\z)
\end{pmatrix}=
\begin{pmatrix}
c_1\\
\vdots\\
c_{k+1}
\end{pmatrix}.
\eeq
where $\bc$ is defined by \eqref{c-vector} above. To streamline notation further, we define a
 $(k+1)\times N$ Vandermonde matrix $\bV_k(x_1,\dots,x_N)$, with $ij$'th
entry equal to $x^{i-1}_j$:

\beq
\bV_k (x_1,\dots,x_N)=
\begin{pmatrix}
1&1&\dots&1\\
x_1&x_2&\dots&x_N\\
x^{2}_{1}&x^{2}_{2}&\dots&x^{2}_{N}\\
\vdots&\vdots&\dots&\vdots\\
x^{k}_{1}&x^{k}_{2}&\dots&x^{k}_{N}\\
\end{pmatrix}.
\eeq

We also define a column vector $\bD(\z) ={ (D_{\bv_1}(\z), \ldots,
D_{\bv_N}(\z))}^\top$, so that (\ref{momentmatrix}) reads
\[\bV_k(x_1,\dots,x_N) \cdot  \bD(\z) = \bc .\]

We may multiply both sides of \eqref{vanderm} on the left by
a row vector $\ba=(a_0,a_1,\ldots,a_k)$.  First, we see that
 \[
\ba\cdot\bV_k(x_1,x_2,\ldots,x_N) = (q_{\ba}(x_1),q_{\ba}(x_2),\ldots,q_{\ba}(x_N)),
\quad\text{where $q_{\ba}(t) = \sum_{\ell=0}^k a_\ell t^\ell$}.
\]
Therefore, taking $\ba$ to be the coefficient vector of the polynomial
\beq
p_{\z}(t)=\prod_{i=1}^N (t-x_i)=\prod_{\bv\in\V(P)}(t-\bil{\bv}{\z})=t^N+\sum_{i=0}^{N-1}a_i t^i,
\eeq
and multiplying  \eqref{vanderm} by $\ba$, we obtain the identity $ 0 = \ba \cdot \bc$. Moreover, for each $0 \le \ell \le k-N$ we substitute for $\ba$ the vector  $\ba_{\ell}$ corresponding to $t^{\ell}p_{\z}(t)$,  to obtain zero in \eqref{vanderm}, when multiplying on the left by $\ba_{\ell}$.   We thus obtain $k-N+1$ equations of the form
\beq\label{easydot}
\ba_{\ell}\cdot\bc=0.
\eeq

As $\ell$ increases, the coefficient vector of $t^{\ell}p_{\z}(t)$ gets shifted
to the right, and it is convenient to capture all of its shifts
simultaneously by the  $m\times m$  Hankel matrix   $\bH:=\bH
(c_1,\dots,c_{2m-1})$, where we fix $m\ge N+1$, defined by:

\beq
\bH (c_1,\dots,c_{2m-1})=
\begin{pmatrix}
c_1&c_2&\dots&c_m\\
c_2&c_3&\dots&c_{m+1}\\
\vdots&\vdots&\dots&\vdots\\
c_m&c_{m+1}&\dots&c_{2m-1}
\end{pmatrix}.
\eeq

\begin{theorem} \label{hsol}
The Hankel matrix $\bH$ has rank $N$ and
its kernel is spanned by the  $m-N$ linearly independent vectors
\beq
\ba_\ell=(\underbrace{0,\dots,0}_{\ell\ \text{ times}},a_0,\dots,a_{N-1},1,
\underbrace{0,\dots,0}_{m-N-1-\ell}), \quad 0\leq\ell \leq m-N-1,
\eeq
\end{theorem}
\proof
For each $\ell$ in the range  $0\le\ell\le 2m-2-N$ we use
\eqref{easydot}, with a vector $\ba_{\ell}$ of length $2m-1$.  Putting all of
these equations together in a more compact form, we can write
$\ba_{\ell}\bH=0$, where now the length of $\ba_{\ell}$ is $m$.

It remains to show that the vectors $\ba_\ell$ generate the full kernel $\K$ of
$\bH$.  Let $\bb'\in\K$.  Without loss of generality,  there exists   an $L<N$
such that $\bb'=(b_0,\dots,b_{L-1},1,0,\dots,0)$, because we may use the
various vectors $\ba_\ell$, which lie in the kernel of $\bH$ to get the
appropriate zeros in this $\bb'$ vector.

Now let us define a number of row vectors of the size $2m-1$:

\beq
\bb_\ell=(\underbrace{0,\dots,0}_{\ell\ \text{ times}},b_0,\dots,b_{L-1},1,
\underbrace{0,\dots,0}_{2m-2-L-\ell}), \quad 0\leq\ell \leq 2m-L-2.
\eeq

By definition of the Hankel matrix and since $m>L+1$, we have
$\bb_\ell\cdot\bc=0$.  Consider the polynomial
$p_\bb(t)=b_0+b_1t+\ldots+b_{L-1}t^{L-1}+t^L$ corresponding to $\bb_0$.

Taking $k=2m-2$ in \eqref{vanderm}, we multiply both sides of \eqref{vanderm} on the left by $\bb_\ell$.
Hence, we get $\bb_\ell\cdot\bV_{2m-2}(x_1,\dots,x_N)\cdot\bD=0$. Therefore, for every $0\leq\ell \leq 2m-L-2$ we have
\beq
(x_{1}^{\ell}p_{\bb}(x_1),\ldots,x_{N}^{\ell}p_{\bb}(x_N))\cdot\bD=0.
\eeq

Combining the first $N$ of the latter equations into a matrix form (note that $N-1<2m-L-2$) we get:

\[
\begin{pmatrix}
p_{\bb}(x_1)&\dots&p_{\bb}(x_N)\\
x_1p_{\bb}(x_1)&\dots&x_Np_{\bb}(x_N)\\
x_{1}^{2}p_{\bb}(x_1)&\dots&x_{N}^{2}p_{\bb}(x_N)\\
\vdots&\dots&\vdots\\
x_{1}^{N-1}p_{\bb}(x_1)&\dots&x_{N}^{N-1}p_{\bb}(x_N)\\
\end{pmatrix}
\begin{pmatrix}
D_{\bv_1}(\z)\\
\vdots\\
D_{\bv_N}(\z)
\end{pmatrix}=
\begin{pmatrix}
0\\
\vdots\\
0
\end{pmatrix},
\]

which can be rewritten as

\[
\begin{pmatrix}
1&\dots&1\\
x_1&\dots&x_N\\
x_{1}^{2}&\dots&x_{N}^{2}\\
\vdots&\dots&\vdots\\
x_{1}^{N-1}&\dots&x_{N}^{N-1}\\
\end{pmatrix}
\begin{pmatrix}
p_{\bb}(x_1)&0&\dots&0\\
0&p_{\bb}(x_2)&\dots&0\\
\vdots&\vdots&\dots&\vdots\\
0&0&\dots&p_{\bb}(x_N)\\
\end{pmatrix}
\begin{pmatrix}
D_{\bv_1}(\z)\\
\vdots\\
D_{\bv_N}(\z)
\end{pmatrix}=
\begin{pmatrix}
0\\
\vdots\\
0
\end{pmatrix}.
\]

Since $\bV_{N-1}(x_1,\dots,x_N)$ is invertible, we get

\[
\begin{pmatrix}
p_{\bb}(x_1)D_{\bv_1}(\z)\\
\vdots\\
p_{\bb}(x_N)D_{\bv_N}(\z)
\end{pmatrix}=
\begin{pmatrix}
0\\
\vdots\\
0
\end{pmatrix}.
\]

As $L<N$ and $p_{\bb}(t)\neq 0$, we deduce that $x_1,\ldots,x_N$ cannot all be
roots of $p_{\bb}(t)$.  It remains to mention that
$D_{\bv_i}(\z)\neq 0$ for every $1\le i\le N$, by the choice of the vector
$\z$ in general position.  We
therefore arrive at a contradiction.  \qed



Once we construct the kernel of $\bH$, we will pick a vector $\left(a_0, \ldots, a_{N-1}, 1, 0, \ldots, 0   \right)$ in $\K$, and then define
the polynomial
$p_\z(t) = a_0 + a_1t + \ldots + a_{N-1} t^{N-1} + t^N$.
We note that this is the unique vector with the
largest number of zeros on the right (in algebraic terms, all the remaining
vectors in the kernel can be obtained as coefficients of polynomials
in the principal ideal $(p_\z)\subset \C[t]$).
By Theorem~\ref{hsol}, the roots $x_i \,(= \bil{\bv_i}{\z})$
of this polynomial are precisely the projections $\bil{\bv_i}{\z}$
that we are seeking.
\beq\label{eq:pa}
p_\z(t) = a_0 + a_1t + \ldots + a_{N-1} t^{N-1} + t^N=
\prod_{\bv\in \V(P)} (t-\bil{\bv}{\z}).
\eeq

In summary we have proved the following:
\begin{theorem}
\label{summary1}
Given the moments (\ref{eq:defmomz}) for a direction $\z\in\R^d$ in  general position,
all the projections $\bil{\bv}{\z}$, $\bv\in\V(P)$ are the real roots of the univariate polynomial
$p_\z$ defined in (\ref{eq:pa}).
\end{theorem}
Finding the kernel of $\bH$ and then computing the coefficients of $p_\z(t)$
can be done efficiently in polynomial time.
After having computed the projections onto $\z$ of all the vertices,
the next step is to find the
projections on each of the $d$ coordinates of all $N$ vertices of $P$.  However,
there is still an inherent ambiguity in this process because we will not know
from which vertex a specific projection came from.  We resolve this problem in
Section~\ref{sec:algorithm} and also in alternative way in Section~\ref{sec:univar}
by using univariate representations.

\begin{rem}~~
\begin{itemize}
\item[(a)] An analogue of \BrLa\ formula for $d=2$ was known for quite a long time
(see e.g. P.~Davis \cite{MR0167602}), and the system of equations corresponding to
\eqref{vanderm} was solved by what is known as Prony's method, see
e.g. Elad, Milanfar, and Golub \cite{MR1740393,EMG04}.
The solution method described above can also be
considered as a variation of the Prony's method.
\item[(b)] Importantly, the quantities $D_{\bv_i}(\z)$ play no role
for computing the projections $\bil{\bv_i}{\z})$! This is why we will be able to
extend the present methodology to general convex
polytopes $P$ (i.e., non necessarily simple).
\end{itemize}
\end{rem}
\section{Polynomial density}
\label{sec:polymeasure}

In this section we address the case of non-uniform measures. That is, our moments are now
defined as

\begin{equation}
\label{eq:defmomzrho}
\mu_j (\z) = \int_P  \langle \x, \z \rangle  ^j \rho(x) d\x,
\end{equation}
where the density function $\rho$ is a homogeneous polynomial of fixed known degree $\de$.
We note that, intuitively, if $\rho$ is not a homogeneous polynomial the change of a physical scale
(e.g. meters to centimeters) will cause  complicated changes in the formulas for moments.
Therefore, the case of a homogeneous polynomial measure is a very natural one, and we begin with this
case in order to develop the proper formulas for it.   We then notice, in the next subsection, that the results for the general case of a polytope with any polynomial
density follows exactly the same analysis as the case of  the homogeneous density.

To set notation, we let $P$ be a convex polytope with a density function $\rho(\x)$.  We separate $\rho$ into its homogeneous polynomial pieces, by writing
$\rho(\x)=\sum_{s=0}^{\de}\rho_{s}(\x)$, where $\rho_{s}(\x)$ is a homogeneous polynomial of degree $s$.     We will require the physically natural assumption that
$\rho(x)>0$ for each $x \in P$, and in fact we will only need the assumption $\rho(\bv)\neq 0$ for $\bv\in\V(P).$

We define $\bV_k=\bV_k(\bil{\bv_1}{\z},\ldots,\bil{\bv_N}{\z})=\bV_k(x_1,\ldots,x_N)$, the standard Vandermonde matrix.
We further define the $l$'th derivative of the Vandermonde matrix, namely $\bV_{k}^{(l)}$, whose $ij$'th entry is equal to $(i-1)\cdot(i-2)\cdot\ldots\cdot(i-l)x^{i-1-l}_j$:

\beq
\bV_{k}^{(l)}(x_1,\dots,x_N)=
\begin{pmatrix}
0&0&\dots&0\\
\vdots&\vdots&\dots&\vdots\\
0&0&\dots&0\\
l!&l!&\dots&l!\\
\vdots&\vdots&\dots&\vdots\\
\frac{k!}{(k-l)!}x^{k-l}_{1}&\frac{k!}{(k-l)!}x^{k-l}_{2}&\dots&\frac{k!}{(k-l)!}x^{k-l}_{N}\\
\end{pmatrix}.
\eeq

As mentioned above, we first assume here that $\rho(\x)$ is a homogeneous polynomial of degree $\de$.
However, in the following subsection we will discuss how the following formulas also work in the
more general case of variable but non-homogeneous polynomial density measures.
We recall the moment formulas for variable density, for a simple polytope $P$, from Theorem \ref{th:MomentsFormula.density} in the Appendix:

\beq \label{diff_equation2}
\mu_j (\z)=
\frac{j! (-1)^d}{ (j+d+\de)!}
 \sum_{\bv\in\V(P)}
  \rho \left(\frac{\partial}{\partial z_1},\dots,\frac{\partial}{\partial z_d}\right)
 \bil{\bv}{\z}^{j+d+\de} D_\bv(\z),
\eeq
where
\beq
D_\bv(\z):=
\frac{    | \det K_{\bv}  |   }
{  \prod_{k=1}^d   \langle \w_k(\bv), \z \rangle  },
\eeq
and the identity is valid for each $\z \in \C^d$ such that the denominators in $D_\bv(\z)$ do not vanish.   In addition, we also have the following companion identities:
\beq
0=\rho\left(\frac{\partial}{\partial z_1},\ldots,\frac{\partial}{\partial z_d}\right)
\sum_{\bv\in\V(P)}
 \bil{\bv}{\z}^{j}   D_\bv(\z),
\eeq
for each $ 0 \leq j \leq d + \de -1$.

We now repeat the same procedure of putting the new moment formulas above into matrix form, as in \eqref{c-vector}.  Here, the definition of the vector $\bc$ is only slightly different, namely:

\beq\label{c-vector_m}
\left (c_1, \dots, c_{k+1} \right) = (-1)^d\left(0, \dots, 0, \frac{(d+\de)!}{0!} \mu_0,  \frac{(1+d+\de)!}{1!} \mu_1, \dots, \frac{k! \cdot \mu_{k-d-\de}}{(k-d-\de)!} \right).
\eeq

We arrive at the following interesting matrix ODE for moments with homogeneous polynomial density:

\beq\label{momentmatrix_m}
\rho\left(\frac{\partial}{\partial z_1},\ldots,\frac{\partial}{\partial z_d}\right)
\left[
\begin{pmatrix}
1&1&\dots&1\\
 \langle \bv_1, \z \rangle &  \langle \bv_2, \z \rangle & \dots &  \langle \bv_N, \z \rangle   \\
{ \langle \bv_1, \z \rangle}^2 & { \langle \bv_2, \z \rangle}^2 & \dots &  {\langle \bv_N, \z \rangle}^2  \\
\vdots&\vdots&\dots&\vdots \\
  { \langle \bv_1, \z \rangle}^k & { \langle \bv_2, \z \rangle}^k & \dots &  {\langle \bv_N, \z \rangle}^k  \\
\end{pmatrix}
\begin{pmatrix}
D_{\bv_1}(\z)\\
\vdots\\
D_{\bv_N}(\z)
\end{pmatrix}\right]=
\begin{pmatrix}
c_1\\
\vdots\\
c_{k+1}
\end{pmatrix},
\eeq

\noindent
where the differentiation is taken separately for each entry of the vector
on the left hand side.

One may check that a single partial derivative of a matrix product obeys the same rule
as the derivative of a product of two functions, that is
$\frac{\partial}{\partial x} (M_1(x)\cdot M_2(x))=
\frac{\partial}{\partial x} M_1(x)\cdot M_2(x))+ M_1(x)\cdot \frac{\partial}{\partial x} M_2(x).$

We compute a partial derivative $\frac{\partial }{\partial z_i}$ of $\bV_{k}(\bil{\bv_1}{\z},\ldots,\bil{\bv_N}{\z})$:

\begin{align*}
\frac{\partial }{\partial z_i}\bV_{k} &=
\begin{pmatrix}
0&\dots&0\\
\bv_{1}^{(i)} & \dots &  \bv_{N}^{(i)} \\
2\bv_{1}^{(i)}\langle \bv_1, \z \rangle  & \dots & 2\bv_{N}^{(i)}\langle \bv_N, \z \rangle \\
\vdots&\dots&\vdots \\
k\bv_{1}^{(i)}{\langle \bv_1, \z \rangle}^{k-1} & \dots &  k\bv_{N}^{(i)}{\langle \bv_N,\z\rangle}^{k-1} \\
\end{pmatrix}  \\
&=
\bV_{k}^{(1)}
\cdot
\begin{pmatrix}
\bv_{1}^{(i)} & 0 & \dots & 0\\
0 & \bv_{2}^{(i)} & \dots &  0 \\
\vdots&\vdots&\dots&\vdots \\
0 & 0 & \dots &  \bv_{N}^{(i)}\\
\end{pmatrix}.
\end{align*}

By repeating the partial derivative in each variable $z_i$, we arrive at:

\beq
\rho\left(\frac{\partial}{\partial z_1},\dots,\frac{\partial}{\partial z_d}\right)\bV_{k}
=
\bV_{k}^{(\de)}
\cdot
\begin{pmatrix}
\rho(\bv_{1}) & 0 & \dots & 0\\
0 & \rho(\bv_{2}) & \dots &  0 \\
\vdots&\vdots&\dots&\vdots \\
0 & 0 & \dots &  \rho(\bv_{N})\\
\end{pmatrix}.
\eeq

Now expanding the matrix ODE formula \eqref{momentmatrix_m}, and using the product rule for differentiation of matrices, we may write it in the following form:

\beq\label{vanderm_m}
\sum_{i=0}^{\de}\bV_{k}^{(i)}
\cdot
\begin{pmatrix}
f_{1}^{(i)}(\z)\\
\vdots\\
f_{N}^{(i)}(\z)
\end{pmatrix}
=
\begin{pmatrix}
c_1\\
\vdots\\
c_{k+1}
\end{pmatrix},
\eeq
where each entry $f_{j}^{(i)}(\z)$ is a rational function of $\z$, and the highest vector term, comprised of the rational functions $f_{j}^{(\de)}(\z)$, has the nice form
\beq\label{highest_t}
\begin{pmatrix}
f_{1}^{(\de)}(\z)\\
\vdots\\
f_{N}^{(\de)}(\z)
\end{pmatrix}
=
\begin{pmatrix}
\rho(\bv_{1}) & 0 & \dots & 0\\
0 & \rho(\bv_{2}) & \dots &  0 \\
\vdots&\vdots&\dots&\vdots \\
0 & 0 & \dots &  \rho(\bv_{N})\\
\end{pmatrix}
\cdot
\begin{pmatrix}
D_{\bv_1}(\z)\\
\vdots\\
D_{\bv_N}(\z)
\end{pmatrix}.
\eeq

For the proof of the following theorem we construct a certain vector as follows.   Define the polynomial
\beq
\label{case2}
p_{\z}(t)=\prod_{\bv\in\V(P)}(t-\bil{\bv}{\z})^{\de+1}=t^{N(\de+1)}+\sum_{i=0}^{(\de+1)N-1}a_i t^i,
\eeq
We define the vector  $\ba_{\ell}$ to be the coefficient vector of the polynomial $t^{\ell} p_{\z}(t)$.

\begin{theorem}
The Hankel $m\times m$ matrix $\bH$, with $m\geq (\de+1) N + 1$ corresponding to the moment formulas with variable density,
has rank $(\de +1)N$, and its kernel is spanned by the linearly independent vectors $\ba_{\ell}$.
\end{theorem}

\proof
We  repeat the procedure that we used in Section   \ref{sec:project}, using a corresponding $m\times m$ Hankel matrix and its kernel,
but this time the dimension is $m \geq (\de+1) N + 1$.
Similarly to the case of uniform density $\rho(x)=1$, we may again multiply both sizes of \eqref{vanderm_m}
on the left by a row vector $\ba_0=(a_0,a_1,\ldots,a_k)$.  First, we see that
\[
\ba_0\cdot\bV_{k}^{(i)}(x_1,x_2,\ldots,x_N) = (p_{\z}^{(i)}(x_1),p_{\z}^{(i)}(x_2),\ldots,p_{\z}^{(i)}(x_N)),
\]
where $p_{\z}^{(i)}(t)$ is $i$'th derivative of $p_{\z}(t).$

Now, multiplying  \eqref{vanderm_m} by $\ba_0$, we obtain the identity $0 = \ba_0 \cdot \bc$.
Similarly, for each $0 \le \ell \le k-N(\de+1)$, we substitute for $\ba_0$ the vector  $\ba_{\ell}$
corresponding to $t^{\ell}p_{\z}(t)$, to obtain $0 = \ba_{\ell} \cdot \bc$.  Hence the vector $\ba_0$ lies in the kernel of $\bH$.

On the other hand, we now claim that no other vector $\bb$ different from those spanned by $\ba_{\ell}$ could be in $\K$. If, contrary to hypothesis, we could find such a vector $\bb$, we may assume without loss of generality that $\bb=(b_0,\dots,b_{l},1,0,\dots,0)$, with $l<(\de+1)N-1$, by reducing it with appropriate linear combinations of $\ba_{\ell}$.

Recall that $\bc=(c_1,\dots,c_{k+1})$, where $k\ge 2m-2\ge 2(\de+1)N.$ Let us consider polynomial $p_{\bb}(t)$
with coefficients of $\bb$. Now let $\bb_{\ell}$ corresponds to the polynomials $t^{\ell}p_{\bb}(t)$, for $0\le\ell\le m$.
Then we have $\bb_{\ell}\cdot \bc=0$, because $\bb$ is in the $\K$ and each vector $\bb_{\ell}$ has the same entries
as $\bb$ only shifted by $\ell$ to the right.

Since a degree of $p_{\bb}(t)$ is smaller than that of $p_{\z}(t)$, we have $p_{\z}(t)\nmid p_{\bb}(t)$. Therefore, there exists
$\bv\in\V(P)$ such that $(t-\bil{\z}{\bv})^{\de+1}\nmid p_{\bb}(t)$.   Without loss of generality, we may assume that $\bv = \bv_1$.
We now construct a polynomial $q(t)$ by multiplying $p_{\bb}(t)$ by sufficiently many linear factors of the form $(t-\bil{\z}{\bv})$,
where $\bv$ varies over all of the vertices of $\V(P)$.  We will treat the particular vertex $\bv_1$ differently, by multiplying by a slightly different power of
 $(t-\bil{\z}{\bv_1})$, to insure that a certain derivative, explicated below, does not vanish at $x_1$, thus giving us a nonzero vector in the kernel of $\bf H$.
 The desired polynomial $q(t)$ satisfies the following properties:

\begin{enumerate}
\item $p_{\bb}(t)|q(t)$.
\item \label{item:prop2}$ (t-\bil{\z}{\bv_1})^{\de+1}\nmid q(t)$.
\item $(t-\bil{\z}{\bv_1})^{\de}\mid q(t)$.
\item $(t-\bil{\z}{\bv})^{\de+1}\mid q(t)$, for $\forall\bv\in\V(P):\bv\neq \bv_1$.
\item $\deg(q)\le\deg(p)+N(\de+1).$
\end{enumerate}

We now write the coefficients of polynomial $q(t)$ as a vector $\bb^o$.   Next, we multiply
\eqref{vanderm_m} on each side by the row vector $\bb^o$.

\beq
\sum_{i=0}^{\de}\bb^o\cdot\bV_{k}^{(i)}
\cdot
\begin{pmatrix}
f_{1}^{(i)}(\z)\\
\vdots\\
f_{N}^{(i)}(\z)
\end{pmatrix}
=
\bb^o\cdot
\begin{pmatrix}
c_1\\
\vdots\\
c_{k+1}
\end{pmatrix},
\eeq

The vector $\bb^o$ may be represented as a linear combination of vectors $\bb_{\ell}$, where $0\le\ell\le m.$
Therefore, we get $\bb^o\cdot\bc=0.$

On the other hand, since $\prod_{i=1}^{N}(t-x_i)^{\de}|q(t)$, we have $\bb^o\cdot\bV_{k}^{\ell}=0$
for each $0\le\ell<\de$.  Then

\[
\bb^o\cdot\bV_{k}^{(\de)}=\left(q^{(\de)}(x_1),\dots,q^{(\de)}(x_N)\right),
\]
where $x_j =  \bil{\z}{\bv_j}$.    We have $q^{(\de)}(x_1)=\gamma\neq 0$, by property (\ref{item:prop2}), and $q^{(\de)}(x_i)=0$ for each $2\le i\le N$, because
$\prod_{i=2}^{N}(t-x_i)^{\de+1}|q(t)$.
Therefore, we get

\beq
\left(\gamma, 0, \dots, 0\right)\cdot
\begin{pmatrix}
\rho(\bv_{1}) & 0 & \dots & 0\\
0 & \rho(\bv_{2}) & \dots &  0 \\
\vdots&\vdots&\dots&\vdots \\
0 & 0 & \dots &  \rho(\bv_{N})\\
\end{pmatrix}
\cdot
\begin{pmatrix}
D_{\bv_1}(\z)\\
\vdots\\
D_{\bv_N}(\z)
\end{pmatrix}
=0.
\eeq

Thus $\gamma\cdot \rho(\bv_1)\cdot D_{\bv_1}(\z)=\bb^o\cdot\bc=0,$ where none of the quantities $\gamma,\rho(\bv_1)$ and $D_{\bv_1}(\z)$
is zero, so that we have arrived at a contradiction.
\hfill $\square$

Therefore, we have proved the following result, the analogue of Theorem~\ref{summary1} for the homogeneous polynomial density
case.
\begin{theorem}
\label{summary2}
Given moments (\ref{eq:defmomzrho}) for a
direction $\z\in\R^d$ in general position and where $\rho$ is a unknown homogeneous polynomial of degree $d^0$, all projections $\bil{\bv}{\z}$, $\bv\in\V(P)$,  are the real roots of the univariate polynomial
$p_\z$ defined in (\ref{case2}).
\end{theorem}

\bigskip
\subsection{Non-homogeneous measure}
We start with the moment formulas for a polytope with variable, but homogeneous density, namely \eqref{diff_equation1}.
Now we let $\de$ be the maximal degree of the monomials of $\rho(x)$.   Then the formula \eqref{diff_equation2}
can be rewritten as follows.

\beq
\label{diff_equation3}
\sum_{s=0}^{\de}\rho_s\left(\nabla\z\right)t^{-s}
\sum_{\bv\in\V(P)}\sum_{j=0}^{\infty}\frac{\bil{\bv}{\z}^j}{j!}(-1)^{d}D_\bv(\z) t^{j-d} =
\sum_{j=0}^{\infty}\frac{\mu_j}{j!}t^j.
\eeq

Following the same reasoning that was used for the homogeneous variable density case, we first collect all the coefficients  of $t^{j-d-\de}$ on both sides of \eqref{diff_equation3}, to get:

\beq
\sum_{s=0}^{\de}\rho_s\left(\nabla\z\right)
\sum_{\bv\in\V(P)}\frac{j(j-1)\ldots(j-\de+s+1)}{j!}\bil{\bv}{\z}^{j-\de+s} \cdot D_\bv(\z)
=
\frac{c_{j+1}}{j!},
\eeq
where $c_{j+1}=(-1)^{d}\frac{j! \cdot \mu_{j-d-\de}}{(j-d-\de)!}$, as in the formula \eqref{c-vector_m}.
Next, we put everything into a matrix form and get

\beq
\sum_{s=0}^{\de}\rho_s(\nabla\z)
\cdot\left[\bV_k^{(\de-s)}(x_1,\dots,x_d)\cdot\bD\right]
=
\begin{pmatrix}
c_1\\
\vdots\\
c_{k+1}
\end{pmatrix}.
\eeq

The latter matrix ODE can be brought into the same form as \eqref{vanderm_m}, with exactly the
same coefficient of  $\bV_{k}^{\de}(x_1,\dots,x_N)$ that appears in \eqref{highest_t}.
Therefore, our method works for general polynomial density measures as well, with precisely the same algorithm.

\section{General convex polytopes}
\label{sec:non_simple}

In the previous discussion we considered only simple polytopes,
because the  \BrLa\ formula takes a particularly nice simple form when $P$ is a simple polytope.
However, it is natural to extend our
approach to non-simple polytopes. Indeed, it is always possible to
triangulate $P$, that is decompose $P$ into a union of non
overlapping simplices, without adding any extra
vertices (See, for example,  \cite[Theorem~3.1]{BR07}).

We now fix one such triangulation of $P$,  and denote it by $\bT(P)$.
We may then rewrite the formula for each moment $\mu_j (\z)$ as follows.

\beq\label{triang}
\mu_j (\z) = \int_P  \langle \x, \z \rangle ^j
d\x = \sum_{\Delta\in\bT(P)}\int_\Delta  \langle \x, \z \rangle ^j
d\x.
\eeq

Triangulating the general convex polytope $P$ into simplices, we reduce the general moment problem to the moment problem for each simplex $\Delta$ of the triangulation.   Although triangulations may be expensive to construct in practice, we only need to consider a theoretical non-vanishing result, given in Lemma \ref{nonvanish} below, for any such triangulation.  Given such a triangulation, we may then apply the formulas (\ref{brion}) and (\ref{brionzero}) to each of the simplices $\Delta$
in the equation above:

\begin{align*}
\bc_j(\z)  &=\sum_{\Delta\in\bT(P)}  \quad \sum_{\bv\in\V(\Delta)} \bil{\bv}{\z}^{j} D_{\bv}(\Delta,\z) \\
&=\sum_{\bv\in\V(P)}\bil{\bv}{\z}^{j}   \sum_{\Delta\in\{\bT(P)|\bv\in\V(\Delta)\}}   D_{\bv}(\Delta,\z),
\end{align*}
where we have interchanged the order of summation in the last equality above.
We now define $\tilde{D}_{\bv}(\z)$ for this fixed triangulation $\bT(P)$ by:
\beq\label{eq:nonsimplenumerator}
\tilde{D}_{\bv}(\z):=\sum_{\Delta\in\{\bT(P)|\bv\in\V(\Delta)\}}
D_{\bv}(\Delta,\z)
\eeq
Then we have
\beq \bc_j(\z)=\sum_{\bv\in\V(P)}\bil{\bv}{\z}^{j}
\tilde{D}_{\bv}(\z). \eeq
This gives us
\beq
\frac{(j+d)!(-1)^d}{j!}\mu_j(\z)=\sum_{\bv\in\V(P)}\bil{\bv}{\z}^{j+d}
\tilde{D}_{\bv}(\z). \eeq
Note, that in Section \ref{sec:project} we never used the
explicit formula for $D_{\bv}(\z)$. The only fact we 
exploited was that $D_{\bv}(\z)\neq 0$ for a general position vector $\z$.
Therefore, we can apply the same approach for non-simple polytopes,
if we are able to prove that $\tilde{D}_{\bv}(\z)\neq 0$ for a general position
vector $\z$.

\begin{lemma}\label{nonvanish}
For any vertex $\bv\in\V(P)$, any fixed triangulation $\bT(P)$ and a general position vector $\z$
we have $\tilde{D}_{\bv}(\z)\neq 0$.
\end{lemma}
\begin{proof}
We begin by noting that $\tilde{D}_{\bv}(\z)$ is a finite linear combination of
 rational functions of $\z$. In fact, according to the
Lemma 8.3 and Chapter 9 of~\cite{MR2455889}, $\tilde{D}_{\bv}(\z)$
is a rational function that is the analytic continuation, in $\z$, of the function

\[
\hat{1}_{K_{\bv}-\bv}(\z)=\int_{K_{\bv}} e^{\bil{\z}{x-\bv}}d\x,
\]
when this integral converges.  We define the dual cone to $K_{\bv} - \bv$ as follows:
$K_\bv^* := \{ \y \in \R^d \mid \langle \y, \x \rangle < 0, \text{ for all } \x \in  K_{\bv} - \bv   \}$.  Indeed, the latter integral converges for
all $\z$ lying in the interior of the dual cone $K_{\bv}^*$. 
Since $K_{\bv}$ is a tangent cone of a convex polytope, the dual cone $K_{\bv}^*$
is non-empty. Clearly $e^{\bil{\z}{x}}$ is positive for all $x\in K_{\bv} - \bv$, if $z \in K_{\bv}^*$.
We obtain the result that $\hat{1}_{K_{\bv}-\bv}(\z)>0$ for all such $z$, and we may therefore conclude
that the analytic continuation
of $\hat{1}_{K_{\bv}-\bv}(\z)$ cannot vanish.
\end{proof}

\section{An exact algorithm}
\label{sec:algorithm}

In Section \ref{sec:project} we have learned how to find the projections
of vertices of $P$ onto a general position axis $\z$. A
short summary of the procedure for such a randomly picked
$\z\in \R^d$ is as follows:

\begin{algorithm}[H]
\begin{enumerate}

\item  Given $2m-1 \ge 2N+1$ moments $c_1,\dots,c_{2m-1}$ for $\z$, construct \\
\noindent a square Hankel matrix $\bH(c_1,\dots,c_{2m-1}).$
\medskip
\item Find the vector $v=\left(a_0, \ldots, a_{M-1}, 1, 0, \ldots, 0 \right)$ in $\K$ \\
\noindent with the minimal possible $M.$ It turns out that the number of \\
       \noindent vertices $N=M$.
\medskip
\item The set of roots $\{x_i(\z) = \bil{\bv_i}{\z} |\bv_i\in\V(P)\}$ of polynomial
      $p_\z(t) = a_0 + a_1t + \ldots + a_{N-1} t^{N-1} + t^N$
       then equals the set of \\
       \noindent projections of $\V(P)$ onto $\z$.
\end{enumerate}
\caption{Computing projections.}\label{fig:compproj}
\end{algorithm}

\begin{rem}
Note that $N$ is an essential part of the input. One cannot rule out existence
of another polytope $P'$ with $|\V(P')|>N$ and the same moments, up to certain degree.
\end{rem}

\begin{rem}
If we work in the context of exact measurements, with rational vertices and rational choices of
$\z$ vectors, then $p_\z$ has only rational roots.  In this rational
context, we may analyze the complexity issues involved by using the
LLL-algorithm due to Lenstra, Lenstra, and Lov\'{a}sz~\cite{MR682664}, because now
the rational roots of $p_\z$ can be found in time which is polynomial in $N$ and in the
bitsize of $\V(P)$.

\end{rem}

In this section, we describe below an exact algorithm to compute $\V(P)$  that runs in
polynomial time given the latter assumptions.
When the roots of $p_\z$ are not available exactly, the algorithm still works,
producing approximate results. However, it seems nontrivial to control the
precision of root-finding, as we need to find the roots of $d$ univariate
polynomials. In Section~\ref{sec:univar} we present a different procedure,
where, in contrast, roots of only one polynomial parametrize $\V(P)$, and
which conceivably is more robust against numerical errors.

We use the assumption that $\z$ is in general position (it suffices to require
that $z$ is not perpendicular to the lines $uv$, for $u,v\in\V(P)$) to maintain
bijectivity of projection onto $\z$, as well as to avoid
division by zero in the terms $D_{v_i}(\z)$.  Choosing $\z$ at random from the
Guassian distribution on
$\R^d$, we get a $\z$ in general position with probability $1$.
Further, to reconstruct the locations of $\V(P)$ given the
projections of vertices on a number of axes we match all projections
of the same vertex as follows.\\


\begin{itemize}
\item Take $d$ linearly independent vectors $\z_1,\dots,\z_d$, each chosen in general position.
\item For every $2\le i\le d$ match projections of $\V(P)$ onto $\z_i$ with projections
      onto $\z_1$.
      \begin{enumerate}
      \item Pick a general position vector $\z=\alpha\z_1+\beta\z_i$
          in the plane generated by $\z_1$ and $\z_i$.
      \item Compute the coefficients of the polynomial $p_{\z}(t)$
            using extra $2N+1$ moments in direction $\z$.
      \item \label{step:match} For each pair of projections $x_j(\z_1),x_k(\z_i)$ onto
            $\z_1$ and $\z_i$ match them whenever $p_{\z}(\alpha x_j+\beta x_k)=0$,
            for $1\le j,k\le N$.
      \item With probability $1$ all vertices will be matched correctly, that is
            $x_k(\z_i)$ is matched with $x_k(\z_1)$.
      \end{enumerate}
\item For each $1\le k\le N$ reconstruct $v_k\in\V(P)$ from its projections $x_k(\z_i)$ for $1\le i\le d.$
\end{itemize}

Indeed, the degree $N$ polynomial
\[ p_\z(t)=\prod_k \left(t-\alpha x_k(\z_1)-\beta x_k(\z_i)\right)\]
has $N$ distinct roots.
We evaluate it at the $N^2$ values $\alpha x_j(\z_1)+\beta x_\ell(\z_i)$.
With probability 1, by choice of $\alpha$ and $\beta$,
$p_\z$ will only vanish when $x_j(\z_1)$
and $x_\ell(\z_i)$ correspond to the projections of the same vertex. \\
(In fact, this part is easy to de-randomize: fixing $\alpha=1$ and
choosing more that $N^3$ different values of $\beta$ gives one a ``good''
pair $\alpha$, $\beta$.)

Note that in total we have used $(2d-1)(2N+1-d)$ distinct moments,
while the description of vertices of $P$ requires $d\cdot N$ real
numbers. That is, our procedure is quite frugal in terms of the
moment's measurements.

As claimed in Main Theorem, we can still improve
on the latter (albeit the corresponding procedure is not polynomial time
any more). Indeed,  we only have moments for $d+1$ directions
$\z_1$,\dots, $\z_d$, $\z=\sum_j \alpha_j \z_j$ in general position, we can still carry out a
similar procedure, although one would need to compute $\binom{N}{d}$ test
values (for all the possible $d$-fold matchings) of
\[ p_{\z_0}(t)=\prod_k(t-\sum_{j=1}^d \alpha_j x_k(\z_j)). \]

\section{An analysis of our algorithm in the rational case}
\label{sect:rational.implementation}

In Section~\ref{sec:algorithm} we described our algorithm under the global assumption
that each direction $\z$ is chosen at random from the continuous domain $\R^d$, thus
getting a general position vector $\z$ with probability $1$. However, in any practical implementation,
all the coordinates of $\z$ have to be rational numbers with bounded denominators
and numerators. In this case the probability that $\z$ does not lie in general position
will be strictly smaller than $1$. In what follows we describe a way to pick our $\z$-directions
and argue that the probability for choosing a ``bad set'' of $\z$-directions
(which are not in general position) is indeed small.

We will always pick our $\z$ vectors to be
rational vectors, with denominator equal to $r$, and lying in the
unit cube $[0, 1]^d$.
If we knew the vertex description of a simple polytope $P$, we would
only need to make sure that $\z$ lies in the complement of the
finite union of hyperplanes that are orthogonal to 
all lines between any two vertices of $P$.
We call such a rational $\z$ a {\em generic} vector.
The probability of picking such a generic $\z$ tends to $1$ as $r
\to \infty$.

We now extend the definition of a generic vector $\z$ to a
non-simple polytope $P$.  In this case, in addition to our
previous restriction that $\z$ is not orthogonal to any line between
vertices of $P$, in particular to the edges of $P$, it might occur
that $\z$ is a zero of the rational function $\tilde{D}_{\bv}(\z)$, defined by 
 \eqref{eq:nonsimplenumerator} in Section \ref{sec:non_simple},
 and we need to avoid such a choice of $\z$.  Hence we define a
 generic vector in the general case of non-simple polytopes to be a vector that
 is simultaneously not orthogonal to any line between vertices of $P$, and also not
 a zero of any rational function $\tilde{D}_{\bv}(\z)$.
In particular, we shall avoid zeros and poles of the complex
function $\tilde{D}_{\bv}(\z)$ in $\z$.

In what follows, we refer to the algorithm of Section~\ref{sec:algorithm}.
By the Schwartz-Zippel Lemma \cite{MR594695,MR575692,DMLip78}, we have an upper bound for the
probability that the numerator and denominator of the
multivariable rational function $\tilde{D}_{\bv}(\z)$ vanishes
for a random rational $\z \in [0, 1]^d$, where $\z$ has denominator $r$.
In fact, by our construction, we have $r^d$ such rational vectors $\z$, and the
Schwartz-Zippel Lemma tells us the following:
for sufficiently large prime $r$ $\pr[\z \text{ is a zero of }
\tilde{D}_{\bv}(\z)] \leq \frac{N}{r}$ and similarly $\pr[\z \text{ is a pole of }
\tilde{D}_{\bv}(\z)] \leq \frac{N}{r}$. Indeed, both the numerator and the denominator of
$\tilde{D}_{\bv}(\z)$ are homogeneous polynomials in $d$ variables $z_1\dots,z_d$ of degree at most $N$
with integer coefficients; none of these polynomials vanish when taken over the finite
field $\F_{r}$, for all sufficiently large $r$.

We remark that our algorithm picks either arbitrary generic vectors
(we pick them uniformly at random from the rational unit cube), or takes an integer
linear combination of two independent random vectors. In the former case
by taking $r$ of order $2^{poly(N,d)}$ one can make the above probabilities for all $\tilde{D}_{\bv}(\z)$
to be negligibly small. In the latter case, we need to be more careful, as the sum of two random vectors
uniformly distributed over the rational unit cube is no longer a random vector distributed
uniformly over the
unit cube. However, we may now consider the vector $\alpha \z_1+\beta \z_i$,
as well as the numerator and denominator of $\tilde{D}_{\bv}(\z)$,
over the finite field $\F_r$. We note that, once we fix $0<\alpha<r$ and $0<\beta<r$, the linear
combination of two independent, uniformly distributed vectors, namely $\alpha \z_1+\beta \z_i$, is again
uniformly distributed over $\F_r^d$.

Therefore, we may assume that each particular direction $\z$
that appeared in the algorithm~\ref{fig:compproj} is generic with a very high probability.
On the other hand, a generic vector $\z=\alpha \z_1+\beta \z_i$ in the plane spanned by $\z_1,\z_i$,
matches the set of projections onto $\z_1$ and the set of projections onto $\z_i$ uniquely at
very high probability.
Indeed, given the projection onto $\z_1$ and $\z_i$ there are $N^2$ possible projections of
$\V(P)$ onto the plane spanned by $\z_1$ and $\z_i$ and at most $N^4$ different lines between these points.
In other words, there are altogether at most $N^4$ directions that do not help us match projections
onto $\z_1$ and $\z_i$. In the algorithm we pick one of the $r$ distinct
directions for $\z=\alpha \z_1+\beta \z_i$ for any fixed $\alpha$. Thus the chance
that our algorithm did make a mistake in a particular step is negligibly small.

\section{Univariate representations for $\V(P)$}
\label{sec:univar}
In this section, we present an alternative procedure, that
is conceivably more robust than the algorithm in Section \ref{sec:algorithm}, where
given a finite collection of projections of the vertices, we presented an exact procedure to reconstruct them.
That is, we were given some data described in Algorithm~\ref{fig:compproj},
assuming that $\V(P)\subset \Q$ and the measurements are exact.
When at least one of the latter assumptions does not hold, the polynomial
$p_\z$, whose roots are projections of $\V(P)$, may not have rational roots.    Even
its coefficients might be known only approximately.  Thus it might be hard to control numerical
errors.

We construct {\em univariate representations} (see e.g.
\cite{BPRbook}) of $\bv\in \V(P)$. The latter are typically used to
compute solutions of systems of multivariate polynomial equations---here
this appears to be the first use of these representations for purposes other than
solving systems of polynomial equations.
That is, we will express the coordinates of $\bv\in \V(P)$ as univariate
rational functions of $\theta$, where $\theta$ is a root of $p_\ba(t)$ in
\eqref{eq:pa}.

We introduce bivariate polynomials $f_{\ba\bb}\in\R[s,t]$ defined by:
\begin{equation}\label{eq:bivar}
(s,t)\mapsto f_{\ba\bb}(s,t)=\prod_{\bv\in \V(P)} (t-\bil{\bv}{\ba+s\bb}),\qquad \ba,\bb\in\R^d.
\end{equation}

Upon transitioning to rational vectors $\ba$ and $\bb$, generic in the sense of
Section~\ref{sect:rational.implementation}, and with $\ba\neq \bb$, we can compute the coefficients
of $f_{\ba\bb}(s,t)$ by interpolating, with respect to $s$, the coefficients
of the polynomials $f_{\ba\bb}(s,t)=p_{\ba+s\bb}(t)$, with $s=0,1,\dots,N$,
and $p_{\ba+s\bb}$ in \eqref{eq:pa} computed using Theorem~\ref{hsol}.
Define
\begin{equation}\label{eq:gab}
g_{\ba\bb}(t):=\frac{\partial f_{\ba\bb}(s,t)}{\partial s}\mid_{s=0}.
\end{equation}
Then
$$g_{\ba\bb}(t)=-\sum_{\bv\in \V(P)}\bil{\bv}{\bb}\prod_{\bv\neq \bu\in \V(P)}(t-\bil{\bu}{\ba}).$$
In particular for $\w\in \V(P)$ one obtains
\begin{eqnarray*}
g_{\ba\bb}(\bil{\w}{\ba})&=&
-\sum_{\bv\in \V(P)}\bil{\bv}{\bb}\prod_{\bv\neq \bu\in \V(P)}\bil{\w-\bu}{\ba}\\
&=&-\bil{\w}{\bb}\prod_{\w\neq \bu\in \V(P)}\bil{\w-\bu}{\ba}.
\end{eqnarray*}
On the other hand, for $p_\ba$ in \eqref{eq:pa}, its derivative $p'_\ba$ reads
$$p'_\ba(t)= \sum_{\bv\in \V(P)}\prod_{\bv\neq \bu\in \V(P)}(t-\bil{\bu}{\ba})$$
and thus
$$p'_\ba(\bil{\w}{\ba})= \sum_{\bv\in \V(P)}\prod_{\bv\neq\bu\in \V(P)}\bil{\w-\bu}{\ba}=
\prod_{\w\neq \bu\in \V(P)}\bil{\w-\bu}{\ba}.$$
Hence
$$\bil{\w}{\bb}=
\frac{g_{\ba\bb}(\bil{\w}{\ba})}{p'_\ba(\bil{\w}{\ba})}=
\frac{g_{\ba\bb}(\theta)}{p'_\ba(\theta)},\quad\text{for some $\theta$ s.t. }
p_\ba(\theta)=0.$$
In particular, assuming that a set of basis vectors $\be_1,\dots,\be_d$ of $\R^d$
are generic, we obtain
\begin{theorem}\label{th:univar}
The set of vertices of $P$ is given by
\begin{equation}\label{eq:univarrep}
\V(P)=\left\{\left( \frac{g_{\ba\be_1}(\theta)}{p'_\ba(\theta)},\dots,
\frac{g_{\ba\be_d}(\theta)}{p'_\ba(\theta)}\right)\mid
\quad\text{ for each $\theta$ s.t. } p_\ba(\theta)=0\right\},
\end{equation}
provided that $\ba,\be_1,\dots,\be_d\in\R^d$ are `sufficiently general' w.r.t. $P$ -- that is,
 the polynomial $p_\ba(t)$ from \eqref{eq:pa}
and the polynomials $g_{\ba\be_j}(t)$ from \eqref{eq:gab} have no multiple
root. \qed
\end{theorem}

We remark that the assumption of being ``sufficiently general''  in Theorem~\ref{th:univar}
is equivalent to the fact that each of the vectors $\ba,\be_1,\dots,\be_d$ does not lie in the
discriminant varieties of the polynomial $p_\ba(t)$ and the set of polynomials $g_{\ba\be_j}(t)$.

Assuming that computation is done with arbitrary precision, the vertices of $P$
can be obtained by evaluating the vectors of rational functions in $\theta$ at the roots of $p_\ba$,
as in \eqref{eq:univarrep}. Therefore, we have transformed the delicate computations of the roots of the
polynomials
$p_\z$ for all projections onto a number of axis vectors $\z$,
into just one calculation given by \eqref{eq:univarrep}.

We note that here we need to use
$O(dN^2)$ moments, which is typically much less frugal than the method of
Section~\ref{sec:algorithm}, which only uses $O(dN)$ of them.

\begin{rem}
A similar computation of the univariate representation
can be carried out even without the genericity assumptions, when the corresponding univariate polynomials have
multiple roots. See \cite{GrPa04} for details.
\end{rem}

\section{An application to physics}
\label{sect:potential}
Here we discuss an application of our results to a classical problem
of mathematical physics---reconstruction of an object from the potential
of a field that it creates. For concreteness, we limit ourselves to the 3-dimensional
potential of the gravitational field.  The potential function $u(x):=u(x_1,x_2,x_3)$
of the gravitational field $F(x)$ is defined by
\[
F(x)=\nabla u(x).
\]
In turn, for a body $T\subset\R^3$ with density $\rho(x)$
the potential is given by
\[
 u(x)=\int_T \frac{\rho(t)}{\|x-t\|}dt,\quad\text{for any $x\not\in T$}.
\]
A typical physics problem is to reconstruct $T$ and $\rho$ from $u$, i.e.
from the measurements of $u$.
That is, we can assume that
$\|x-t\|^{-1}=\sum_{a} f_a(x)t^a$ is an expansion in a Taylor series
w.r.t. $t=(t_1,t_2,t_3)$, and the $f_a(x)$ depend upon $x$ only.
Then the expansion
\[
 u(x)=\sum_{a} f_a(x)\int_T t^a\rho(t)dt,\quad\text{for any $x\not\in T$}
\]
encodes information of the  moments $\int_T t^a \rho(t)dt$ of the measure
$\rho(t)$ supported on $T$.  Thus reconstructing $T$ and $\rho$ from $u$ is an
inverse moment problem. For instance, when $\rho$ is a polynomial and $T$ is a
polytope, the approach described in this paper can be applied to this inverse
potential problem and will provide an exact reconstruction.

 \bigskip
\paragraph{\textbf{Acknowledgments.} We thank the referee
for very useful suggestions, which indeed improved the text.}

\nocite{MR2199920,MR2589247}
\bibliographystyle{alpha}
\bibliography{poly,master}

\newpage
\appendix
\section{Proof of the \BrLa\ identities for moments of polytopes}
\label{LawrenceProof}

Here we recall the proof of the moment formulas of \BrLa, as well as some useful facts that arise in combinatorial geometry,
and which we have used in the present paper,
but which may not be well-known yet to the mathematical community at large.  Although some of the proofs here may not be completely self-contained, they give the reader the proper background for understanding where the moment formulas come from, and the tools that are used for handling them.   For more detailed proofs of some of these results, the reader may consult the book \cite{MR2455889}, as well as the book \cite[Corollary~11.9]{BR07}.
We begin with a very useful  geometric identity, which has an
inclusion-exclusion structure, due to Brianchon and Gram.  

We let $1_P(x)$ denote the indicator function of any convex polytope $P$.  For any
$d$-dimensional convex polytope $P$, we have the following Brianchon-Gram
identity:
\beq\label{eq:BriaGra}
1_P (\x) = \sum_{F \subset P} (-1)^{\dim(F)} 1_{K_F}(\x),
\eeq
valid for all $\x \in \R^d$.  Here we are using the tangent cone $K_F$ at each face $F$ of $P$.


\begin{lemma}\label{lem:hatcone_charfun}
\[
\hat{1}_P (\x) = \sum_{v \in \V(P)} \hat{1}_{K_\bv}(\x),
\]
\end{lemma}

\proof
We simply take the Fourier-Laplace transform of both sides of the
Brianchon-Gram identity above, and we recall that it is defined by $ \hat{f}(\z) :=     \int_{\R^d} f(\x)e ^{ \langle \x, \z \rangle} d\x$, valid for all $\z \in \C^d$ for which the integral converges.
By definition, we have
\[
\hat{1}_P(\z) = \int_{P} e ^{ \langle \x, \z \rangle} d\x,
\]
the Fourier-Laplace transform of the indicator function of $P$.  It turns out that we may define the Fourier-Laplace transform
$\hat{1}_{K_F} (x) = 0$, for any tangent cone $K_F$ which contains a line (isomorphic to $\R^1$).
Since all tangent cones $K_F$ contain a line,  except for the vertex tangent cones, we are left only with the
Fourier-Laplace transforms of the vertex tangent cones.   Precisely, we get:
\[
\hat{1}_P (\x) = \sum_{v \in \V(P)} \hat{1}_{K_\bv}(\x). \qedhere
\]

Using the theory of valuations, one can make the proof of the former Lemma more rigorous (see \cite{MR2455889}).  However, for the purposes of this appendix, 
it is not necessary to consider the subtle issues of convergence that arise here. 

\begin{lemma} \label{lem:hatcone}
Let $K_\bv$ be a vertex tangent cone of a simple polytope $P$.  Then
\[
\hat{1}_{K_\bv}(\z) = (-1)^d  \ \frac{  e ^{ \langle \bv, \z \rangle}  \det K_{\bv}     }{  \prod_{k=1}^d   \langle \w_k(\bv), \z \rangle  },
\]
for all $\z \in \C^d $ such that the denominator does not vanish.
\end{lemma}

\proof
The main idea here is to use the fact that there is a linear transformation that maps the simple tangent cone $K_\bv$ bijectively onto the positive orthant
 $K_{orth}:= \{ (x_1, \dots, x_d)  \in \R^d  \quad | \quad x_j \geq 0    \}$.   To be explicit, let $K_\bv - \bv := K_0$ be the translated copy of our tangent cone $K_\bv$, so that the vertex of $K_0$ lies at the origin.    Let $M$  be the invertible matrix whose columns are the $d$ linearly independent edge vectors $\w_k(\bv)$ of $K_\bv$.   Then the linear transformation $T: \  K_{orth} \rightarrow K_\bv - \bv$, defined by $T(x) = Mx$, gives us the desired bijection from the positive orthant onto the translated tangent cone $K_\bv - \bv := K_0$.   Now we use the explicit computation for the Fourier-Laplace transform of the positive orthant $K_{orth}$, namely:
 \[
 \hat{1}_{K_{orth}}(\z)=  \prod_{j=1}^d \hat{1}_{\R_{\geq 0}}(z_j) = (-1)^d \prod_{j=1}^d  \left(  \frac{1}{z_j} \right).
 \]

Finally, the standard Fourier identity $\widehat{(f \circ T)}(\z) = |\det T|  \hat f(T^t \z)$, valid for any invertible linear transformation $T$, allows us to finish the computation:
  \begin{align*}
 \hat{1}_{K_{\bv}}(\z) &=    \hat{1}_{K_{0} + \bv}(\z) \\
 &= e^{\langle \bv, \z \rangle}    \hat{1}_{K_{0}}(\z) \\
 &=e^{\langle \bv, \z \rangle}    \hat{1}_{M(K_{orth})}(\z) \\
 & = e^{\langle \bv, \z \rangle}    |\det M| \hat{1}_{K_{orth}}(M^t \z) \\
 &= e^{\langle \bv, \z \rangle}     (-1)^d \det K_{\bv}  \prod_{j=1}^d  \left(  \frac{1}{\langle \w_k(\bv), \z \rangle} \right).
 \end{align*}
\qedhere


\bigskip

\begin{theorem}\label{th:BrionBarvinoketc}
Let $P$ be a simple convex polytope.   An explicit formula for the Fourier-Laplace transform of $P$ is given by:

\beq\label{eq:simpleexp}
\int_{P} e ^{ \langle \x, \z \rangle} d\x = (-1)^d  \sum_{\bv \in \V(P)}
\frac{     e ^{ \langle \bv, \z \rangle}  \det K_{\bv}     }
{  \prod_{k=1}^d   \langle \w_k(\bv), \z \rangle  },
\eeq
for all $\z$ that are not orthogonal to any edge of $P$.
\end{theorem}
\proof
From Lemma \ref{lem:hatcone_charfun}, we know that the Fourier-Laplace transform of $P$ is given
by the sum of the Fourier-Laplace transforms of the vertex tangent cones $K_{\bv}$, over all vertices $\bv$ of $P$.
Using Lemma \ref{lem:hatcone} to rewrite the Fourier-Laplace transform of each vertex tangent cone explicitly, we are done.
\qed


\begin{theorem}\label{th:BrionBarvinokMeasure}
For any convex polytope $P$ and any polynomial $\rho \in\R[\x]$, there exist
rational functions $q_{\bv}(\z)$ such that

\beq\label{eq:BrionBarvinoketc}
\int_{P} e ^{ \langle \x, \z \rangle} \rho(\x)d\x = \sum_{\bv \in \V(P)}
  e^{ \langle \bv, \z \rangle} q_{\bv}(\z),
\eeq
for all $\z$  such that the function  $e^{ \langle \bv, \z \rangle} q_{\bv}(\z)$ is analytic at $\z$.
\end{theorem}

\proof
We may first employ the fact that every convex polytope $P$ has a triangulation into some $M$ simplices $\Delta_i$, with no new vertices.
We therefore have $\hat 1_P(\z) = \sum_{i=1}^M \hat 1_{\Delta_i}(\z)$, because
the $d$-dimensional Fourier transform vanishes on all of the lower-dimensional intersections
of the various simplices $\Delta_i$.
We  observe that
\beq
\int_{P} e ^{\langle \x, \z \rangle} \rho(\x)d\x =
\rho \left(\frac{\partial}{\partial z_1},\dots,\frac{\partial}{\partial z_d}\right)
\int_{P} e ^{\langle \x, \z \rangle} d\x,
\eeq
because due to the compactness of $P$,  differentiation under the integral sign is valid.  Thus
\beq
\rho \left(\frac{\partial}{\partial z_1},\dots,\frac{\partial}{\partial z_d}\right)
\int_{P} e ^{\langle \x, \z \rangle} d\x =
\rho \left(\frac{\partial}{\partial z_1},\dots,\frac{\partial}{\partial z_d}\right)
\sum_{i=1}^M \hat 1_{\Delta_i}(\z).
\eeq

Now by Theorem \ref{th:BrionBarvinoketc}, applied to each simple polytope $\Delta_i$, we finally have
\beq \label{density.operator}
\int_{P} e ^{\langle \x, \z \rangle} \rho(\x)d\x  =
\rho \left(\frac{\partial}{\partial z_1},\dots,\frac{\partial}{\partial z_d}\right)
\sum_{i=1}^M  (-1)^d  \sum_{\bv \in \V(\Delta_i)}
\frac{     e ^{ \langle \bv, \z \rangle}  \det K_{\bv}     }
{  \prod_{k=1}^d   \langle \w_k(\bv), \z \rangle  },
\eeq
giving us the desired conclusion upon applying the differential operator to each rational function.
\qed

\bigskip
\noindent
We recall from the introduction that the basis-free moments for uniform density were defined by
\[
\mu_j (\z):= \int_P  \langle \x, \z \rangle^j dx.
\]

\bigskip
The following set of moment formulas can also be found in \cite[Section~3.2]{MR982338}, as well as in \cite[Section~10.3]{BR07}.

\begin{theorem} \label{th:MomentsFormula} (Moments Formula for uniform density)
Given a simple polytope $P$, with uniform density $\rho \equiv 1$, we have the moment formulas:
\beq\label{brion2}
\mu_j (\z)=
\frac{j! (-1)^d}{ (j+d)!}
\sum_{\bv\in \V(P)} \bil{\bv}{\z}^{j+d} D_\bv(\z),
\eeq
for each integer $j \geq 0$, where
\beq
D_\bv(\z):=\frac{|\det K_\bv|}{\prod_{k=1}^d\bil{\w_k(\bv)}{\z}},
\eeq
for each $\z \in \C^d$ such that the denominators in $D_\bv(\z)$ do not vanish. Moreover, we also have the following companion identities:
\beq\label{brionzero2}
0=\sum_{v\in \V(P)}
\bil{\bv}{\z}^{j} D_\bv(\z),
\eeq
for each $ 0 \leq j \leq d-1$.
\end{theorem}
\proof
We begin with the explicit identity for the Fourier-Laplace transform of any convex polytope, namely \eqref{eq:simpleexp}, and we replace $\z$ by $t\z$, where $t >0$ is now treated as a real variable:
\[
\int_{P} e ^{ t \langle \x, \z \rangle} d\x = (-1)^d  \sum_{\bv \in \V(P)}
\frac{     e ^{t \langle \bv, \z \rangle}  \det K_{\bv}     }
{ t^d \prod_{k=1}^d   \langle \w_k(\bv), \z \rangle  }.
\]
Now we expand both sides in their Laurent series about $t=0$, and equate the coefficient of $t^j$ on both sides to obtain the desired moment identities.
\qed

 \bigskip

\begin{theorem} \label{th:MomentsFormula.density} (Moments Formula for polynomial density and any convex polytope)
Suppose we have a homogeneous polynomial density function $\rho(\x)$, of degree $\de$, defined over any convex polytope $P$.
For each integer $j \geq 0$, we have the density moments formulas
\beq\label{moments.density}
\mu_j (\z)=
\frac{j! (-1)^d}{ (j+d+\de)!}
\sum_{i=1}^M \sum_{\bv\in\V(\Delta_i)}
  \rho \left(\frac{\partial}{\partial z_1},\dots,\frac{\partial}{\partial z_d}\right)
 \bil{\bv}{\z}^{j+d+\de} D_\bv(\z),
\eeq
where
\beq\label{Dvz2}
D_\bv(\z):=
\frac{    | \det K_{\bv}  |   }
{  \prod_{k=1}^d   \langle \w_k(\bv), \z \rangle  }.
\eeq
These identities are valid for each $\z \in \C^d$ such that the denominators in $D_\bv(\z)$ do not vanish.   In addition, we also have the following companion identities:
\beq\label{brionzero3}
0=\rho\left(\frac{\partial}{\partial z_1},\ldots,\frac{\partial}{\partial z_d}\right)
\sum_{i=1}^M \sum_{\bv\in\V(\Delta_i)}
 \bil{\bv}{\z}^{j}   D_\bv(\z),
\eeq
for each $ 0 \leq j \leq d + \de -1$.
\end{theorem}

\proof
We begin with (\ref{density.operator}), and replace $\z$ by $t \z$, for any fixed $t >0$. Again, expanding both sides in their Laurent expansions about $t=0$ gives us:

\begin{eqnarray}  
\sum_{j=0}^{\infty}\frac{\mu_j}{j!}t^j &=& \rho\left(\frac{\partial}{t\cdot\partial z_1},\ldots,\frac{\partial}{t\cdot\partial z_d}\right)\sum_{i=1}^M \sum_{\bv\in\V(\Delta_i)}\sum_{j=0}^{\infty}\frac{\bil{\bv}{\z}^j}{j!}(-1)^{d}D_\bv(\z) t^{j-d} \notag\\
&=&\rho\left(\frac{\partial}{\partial z_1},\ldots,\frac{\partial}{\partial z_d}\right)
\sum_{i=1}^M \sum_{\bv\in\V(\Delta_i)}
\sum_{j=0}^{\infty}\frac{\bil{\bv}{\z}^j}{j!}(-1)^{d}D_\bv(\z) t^{j-d-\de}. 
\label{diff_equation1}
\end{eqnarray}

We now equate the coefficient of $t^j$, for each $j \geq 0$, on both sides of the former identity \eqref{diff_equation1}, to obtain the desired moment formulas for variable density:
\beq\label{brion_m}
\mu_j (\z)=
\frac{j! (-1)^d}{ (j+d+\de)!}
\sum_{i=1}^M \sum_{\bv\in\V(\Delta_i)}
\rho\left(\frac{\partial}{\partial z_1},\ldots,\frac{\partial}{\partial z_d}\right)
 \bil{\bv}{\z}^{j+d+\de} D_\bv(\z).
\eeq
Moreover, we also obtain the desired companion identities (\ref{brionzero3}), by equating the first $d+ \de$ coefficients of \eqref{diff_equation1},
for each $ 0 \leq j \leq d+\de-1$.
\qed

\end{document}